\documentclass[12pt,notitlepage]{amsart}
\usepackage{latexsym,amsfonts,amssymb,amsmath,amsthm} 
\usepackage{color}
\newtheorem{corollary}{Corollary}
\newtheorem{theorem}{Theorem}
\newtheorem{lemma}{Lemma}
\newtheorem{proposition}{Proposition}
\newtheorem{conjecture}{Conjecture}

\theoremstyle{remark}

\theoremstyle{plain}

\theoremstyle{remark}

\numberwithin{equation}{section}
\begin{document}

\title{A Note On The Mean Value of $L$--functions in Function Fields.}

\author{Julio C. Andrade}
\address{School of Mathematics, University of Bristol, Bristol BS8 1TW, UK}
\email{j.c.andrade@bristol.ac.uk}
\thanks{The author is supported by an Overseas Research Scholarship and an University of Bristol Research Scholarship.}



\begin{abstract}
An asymptotic formula for the sum $\sum L(1,\chi)$ is established for a family of hyperelliptic curves of genus $g$ over a fixed finite field $\mathbb{F}_q$ as $g\rightarrow\infty$ making use of the analogue of the approximate functional equation for such $L$--functions. As a corollary, we obtain a formula for the average of the class number of the associated rings $\mathbb{F}_{q}[T,\sqrt{D}]$. 
\end{abstract}

\maketitle

\section{Introduction}
In his Disquisitiones Arithmeticae \cite{5}, Gauss presented two conjectures concerning the average values of the class number associated with binary quadratic forms $ax^{2}+2bxy+cy^{2}$ where $a,b,c\in\mathbb{Z}$. For completeness, clarity and to put our problem in the right context we will restate the Gauss's conjectures. 

Let $D=4(b^{2}-ac)$ be the discriminant of the quadratic form $ax^{2}+2bxy+cy$ with $D\equiv0,1\pmod4$. Recall that two quadratic forms are equivalent if is possible to transform the first form into the second through an invertible integral linear change of variables. So we have defined an equivalence relation on the set of quadratic forms and the equivalence classes will be called classes of quadratic forms. Gauss showed that the number of equivalence classes of quadratic forms with discriminant $D$ is finite. Let $h_{D}$ denote this number, we also call $h_{D}$ the \textbf{class number}.  We now present Gauss's conjectures quoted from \cite{6}

\begin{conjecture}[Gauss]
Let $h_{D}$ be the class number defined as above. So,
\begin{enumerate} 
	\item Let $D=-4k$ run over all negative discriminants with $k\leq N$. Then
	\begin{equation}\label{eq:conj1}
	\sum_{1\leq k\leq N}h_{D}\sim\frac{4\pi}{21\zeta(3)}N^{\tfrac{3}{2}}.
	\end{equation}
	\item Let $D=4k$ run over all positive discriminants such that $k\leq N$. Then
\begin{equation}\label{eq:conj2}
\sum_{1\leq k\leq N}h_{D}R_{D}\sim\frac{4\pi^{2}}{21\zeta(3)}N^{\tfrac{3}{2}}.
\end{equation}
Where the number $R_{D}$ is associated to the regulator of the real quadratic number field $\mathbb{Q}(\sqrt{D})$.
\end{enumerate}
\end{conjecture}

The first conjecture \eqref{eq:conj1} was proved by Lipschitz \cite{8} and the second \eqref{eq:conj2} by Siegel \cite{11}. We now will move to the function field analogue of these results and for this let us now define some notation that will be used in the rest of this paper. For more details see \cite{6,9}.

\subsection{Facts about $\mathbb{F}_{q}[T]$}

The norm of the polynomial $f\in\mathbb{F}_{q}[T]$ is defined to be $|f|=q^{\mathrm{deg}(f)}$, and we define the zeta function associated to $A=\mathbb{F}_{q}[T]$ by
\begin{equation}
\zeta_{A}(s)=\sum_{\substack{f\in A \\ f \ \mathrm{monic}}}\frac{1}{|f|^{s}}, \ \ \ \ \ (\mathfrak{R}(s)>1)
\end{equation}
and is easy to show that
\begin{equation}
\zeta_{A}(s)=\frac{1}{1-q^{1-s}}.
\end{equation}
We also define the analogue of the M\"{o}bius function for $A=\mathbb{F}_{q}[T]$ as follows:
\begin{equation}\label{eq:3.3}
\mu(f)=\left\{
\begin{array}{rcl}
(-1)^{t}, & f=\alpha P_{1}P_{2}\ldots P_{t},\\
0, & \mathrm{otherwise},\\
\end{array}
\right.
\end{equation}
where each $P_{j}$ is a distinct monic irreducible polynomial. And the Euler totient function for $A$, denoted by $\Phi(f)$, is defined to be the number of non--zero polynomials of degree less than $\mathrm{deg}(f)$ and relatively prime to $f$.

We will assume from now that $q\equiv1\pmod4$ and let $\mathbb{F}_{q}$ be a finite field with $q$ elements. We denote by $A=\mathbb{F}_{q}[T]$ the polynomial ring over $\mathbb{F}_{q}$ and by $k=\mathbb{F}_{q}(T)$ the rational function field over $\mathbb{F}_{q}$. Let $K/k$ be a quadratic extension and call $\mathcal{O}_{K}$ the integral closure of $A$ in $K$. 

Consider that $D\in A$ is a square--free polynomial and put $\mathcal{O}_{K}$ as $\mathcal{O}_{D}=A[\sqrt{D}]$. So in this case, $\mathcal{O}_{D}$ is a Dedekind domain and the associated class number $h_{D}$ is equal to $|\mathrm{Pic}(\mathcal{O}_{D})|$, where $\mathrm{Pic}(\mathcal{O}_{D})$ is the Picard group of $\mathcal{O}_{D}$. For a detailed explanation about the Picard group in this context see \cite[pg. 315]{9}. 

Let $D\in A$ be a monic and square--free polynomial. We can define the quadratic character $\chi_{D}$ using the quadratic residue symbol for $\mathbb{F}_{q}[T]$ by
\begin{equation}
\chi_{D}(f)=\left(\frac{D}{f}\right),
\end{equation}
and the associated $L$--function by
\begin{equation}\label{eq:l-function}
L(s,\chi_{D})=\sum_{\substack{f\in A \\ f \ \mathrm{monic}}}\frac{\chi_{D}(f)}{|f|^{s}}.
\end{equation}
Now, using \cite[Proposition 4.3]{9} we have $L(s,\chi_{D})$ is a polynomial in $u=q^{-s}$ of degree at most $\mathrm{deg}(D)-1$. 

Hoffstein and Rosen \cite{5} succeeded in calculating the average value of the class number $h_{D}$ when the average is taken over all monic polynomials $D$ of a fixed degree, they showed that:

\begin{theorem}[Hoffstein and Rosen]
Let $M$ be odd and positive. Then,
\begin{equation}
\frac{1}{q^{M}}\sum_{\substack{D \ \mathrm{monic} \\ \mathrm{deg}(D)=M}}h_{D}=\frac{\zeta_{A}(2)}{\zeta_{A}(3)}q^{\tfrac{M-1}{2}}-q^{-1}.
\end{equation}
\end{theorem}

The theorem above can be seen as the function field analogue of the Gauss's conjectures. A problem which is more difficult and we consider in this paper is to average the class number over fundamental discriminants, i.e., $D$ monic and square--free. We should note that the calculations presented in this paper follows the same philosophy of the calculations firstly presented by Faifman, Kurlberg and Rudnick in \cite{4,7,10}, where we fix the number of elements of the finite field and compute the limit as $\mathrm{deg}(D)\rightarrow\infty$ to obtain our asymptotic formulas. 

\section{Statement of  Results}

The main objective of this paper is to  establish an asymptotic formula for

\begin{equation}
\frac{1}{\#\mathcal{H}_{2g+1,q}}\sum_{D\in\mathcal{H}_{2g+1,q}}h_{D},
\end{equation}
as $\mathrm{deg}(D)\rightarrow\infty$. Where $h_{D}$ is the associated class number and $\mathcal{H}_{2g+1,q}$ denotes the following set,
\begin{equation}\label{eq:ensemble}
\mathcal{H}_{2g+1,q}=\{D \ \mathrm{monic}, \ \mathrm{deg}(D)=2g+1, \ D \ \mathrm{square \ free}, \ D\in A \}.
\end{equation}
The number of elements in $\mathcal{H}_{2g+1,q}$ is given by,
\begin{equation}\label{eq:2.3}
\#\mathcal{H}_{2g+1,q}=(q-1)q^{2g}
\end{equation}
as can be seen using \cite[Proposition 2.3]{9}.

\subsection{The Geometric Viewpoint}

As we said earlier, the class numbers $h_{D}$ is equal to the $|\mathrm{Pic}(\mathcal{O}_{D})|$, where $\mathrm{Pic}(\mathcal{O}_{D})$ is the Picard group of $\mathcal{O}_{D}$. But we should also note that if $D\in\mathcal{H}_{2g+1,q}$, then the equation $y^{2}=D(T)$ defines a hyperelliptic curve $C_{D}$ over $\mathbb{F}_{q}$ of genus $g$ and the number $h_{D}$ is closely related to the set of the $\mathbb{F}_{q}$--rational points on its Jacobian, $\mathrm{Jac}(C_{D})$, and so our result also has a geometric appeal. We will now develop this geometric side and then state our results.

The zeta function of the curve $C_{D}$ over $\mathbb{F}_{q}$ is a rational function as shown by Weil \cite{12},
\begin{equation}
Z_{C_{D}}(u)=\frac{P_{C_{D}}(u)}{(1-u)(1-qu)},
\end{equation}
where $P_{C_{D}}(u)$ is a polynomial of degree $2g$ with coefficients in $\mathbb{Z}$. Making use of \cite[Proposition 14.6 and Proposition 17.7]{9}, we can show that the zeta function of the curve $C_{D}$ is given by
\begin{equation}
Z_{C_{D}}(u)=\frac{\mathcal{L}(u,\chi_{D})}{(1-u)(1-qu)},
\end{equation}
where $\mathcal{L}(u,\chi_{D})=L(s,\chi_{D})$ is the $L$-function associated with the quadratic character as given in \eqref{eq:l-function}. We have $\mathcal{L}(u,\chi_{D})$ satisfies the following functional equation
\begin{equation}
\label{eq:funceq}
\mathcal{L}(u,\chi_{D})=(qu^{2})^{g}\mathcal{L}\left(\frac{1}{qu},\chi_{D}\right),
\end{equation} 
as can be seen from \cite[Theorem 5.9]{9}. This allows us to state and prove a lemma which will be the starting point of the main calculations of this paper.
\begin{lemma}
If $D\in\mathcal{H}_{2g+1,q}$ we have that $L(s,\chi_{D})$ can be expressed as follows
\begin{equation}\label{eq:approx}
L(s,\chi_{D})=\sum_{\substack{f_{1} \ \mathrm{monic} \\ \mathrm{deg}(f_{1})\leq g}}\frac{\chi_{D}(f_{1})}{|f_{1}|^{s}}+(q^{1-2s})^{g}\sum_{\substack{f_{2} \ \mathrm{monic} \\ \mathrm{deg}(f_{2})\leq g-1}}\frac{\chi_{D}(f_{2})}{|f_{2}|^{1-s}}.
\end{equation}
\end{lemma}
\begin{proof}
Using the ideas presented by Conrey \textit{et.al.} in \cite{3} we substitute $\mathcal{L}(u,\chi_{D})=\sum_{n=0}^{2g}a_{n}u^{n}$ into the functional equation \eqref{eq:funceq}
\begin{eqnarray}
\sum_{n=0}^{2g}a_{n}u^{n} & = & q^{g}u^{2g}\sum_{m=0}^{2g}a_{m}\left(\frac{1}{qu}\right)^{m}\nonumber \\
& = & \sum_{m=0}^{2g}a_{m}q^{g-m}u^{2g-m}=\sum_{k=0}^{2g}a_{2g-k}q^{k-g}u^{k}.\nonumber
\end{eqnarray}
Therefore,
$$\sum_{n=0}^{2g}a_{n}u^{n}=\sum_{k=0}^{2g}a_{2g-k}q^{k-g}u^{k}.$$
Equating coefficients we have that
$$a_{n}=a_{2g-n}q^{n-g} \ \ \ \ \ \mathrm{or} \ \ \ \ \ a_{2g-n}=a_{n}q^{g-n}$$
and so we can write the polynomial $\mathcal{L}(u,\chi_{D})$ as
\begin{eqnarray}
\sum_{n=0}^{2g}a_{n}u^{n} & = & \sum_{n=0}^{g}a_{n}u^{n}+\sum_{m=0}^{g-1}a_{2g-m}u^{2g-m}\nonumber \\
& = & \sum_{n=0}^{g}a_{n}u^{n}+q^{g}u^{2g}\sum_{m=0}^{g-1}a_{m}q^{-m}u^{-m}\label{eq:3.22}.
\end{eqnarray}
Writing $\displaystyle{a_{n}=\sum_{\substack{f \ \mathrm{monic} \\ \mathrm{deg}(f)=n}}\chi_{D}(f)}$ and $u=q^{-1/2}$ in \eqref{eq:3.22} proves the lemma.
\end{proof}

With this in mind we are ready to state our main results.

\begin{theorem}\label{thm:mainthm1}
Let $\mathbb{F}_{q}$ be a fixed finite field with $q\equiv1\pmod4$. Then
\begin{eqnarray}
\sum_{D\in\mathcal{H}_{2g+1,q}}L(1,\chi_{D})& =& |D|\left\{P(2)-P(1)\left(\frac{1}{q^{[g/2]+1}}+\frac{1}{\zeta_{A}(2)^{2}q^{g[(g-1)/2]}}\right)\right\}\nonumber\\
& + & O((2q)^{g}),
\end{eqnarray}
where $|D|=q^{2g+1}$ and
\begin{equation}
P(s)=\prod_{\substack{P \ \mathrm{monic} \\ \mathrm{irreducible}}}\left(1-\frac{1}{(|P|+1)|P|^{s}}\right).
\end{equation}
\end{theorem}

As a corollary of the Theorem \ref{thm:mainthm1} we have,
\begin{corollary}\label{cor:cor}
\begin{equation}
\frac{1}{\#\mathcal{H}_{2g+1,q}}\sum_{D\in\mathcal{H}_{2g+1,q}}L(1,\chi_{D})\sim\zeta_{A}(2)P(2)
\end{equation}
as $\mathrm{deg}(D)\rightarrow\infty$, i.e., $g\rightarrow\infty$.
\end{corollary}
\begin{proof}
Using the Theorem \ref{thm:mainthm1} together with \eqref{eq:2.3} and computing the limit as $g\rightarrow\infty$ we can conclude the asymptotic formula above.
\end{proof}

To obtain a formula for the average of the class number $h_{D}$ we will make use of the following theorem due Artin \cite{2}.

\begin{theorem}[Artin]\label{thm:artin}
Let $D\in A$ be a square--free polynomial of degree $M$. Then if $M$ is odd we have,
\begin{equation}
L(1,\chi_{D})=\frac{\sqrt{q}}{\sqrt{|D|}}h_{D}.
\end{equation} 
\end{theorem}

We now have all the ingredients to obtain an asymptotic formula for the average of the class number $h_{D}$. 

\begin{theorem}
Under the same hypothesis as in Theorem \ref{thm:mainthm1} we have,
\begin{equation}
\frac{1}{\#\mathcal{H}_{2g+1,q}}\sum_{D\in\mathcal{H}_{2g+1,q}}h_{D}\sim\frac{\sqrt{|D|}}{\sqrt{q}}\zeta_{A}(2)P(2),
\end{equation}
as $\mathrm{deg}(D)\rightarrow\infty$.
\end{theorem}
\begin{proof}
Straightforward by making use of the Corollary \ref{cor:cor} and Theorem \ref{thm:artin}.
\end{proof}

\section{Preliminary Lemmas}

We will require some auxiliary lemmas. Firstly, we will begin by stating a result due to Faifman and Rudnick \cite{4}. 
\begin{lemma}[Faifman--Rudnick]
\label{lem:RHfunctionfields}
Let $\chi$ be a nontrivial Dirichlet character modulo $D$. Then for $n<\mathrm{deg}(D)$,
\begin{equation}
\Bigg|\sum_{\substack{B \ \mathrm{monic} \\ \mathrm{deg}(B)=n}}\chi(B)\Bigg|\leq\binom{\mathrm{deg}(D)-1}{n}q^{n/2}.
\end{equation}
\end{lemma}
We now prove a bound for non--trivial character sums using the Lemma \ref{lem:RHfunctionfields}, which is a consequence of the Riemann Hypothesis for function fields.

\begin{lemma}\label{lem:nonsquare}
We have that, 
\begin{enumerate}
\item
\begin{equation}\label{eq:non1}
\sum_{D\in\mathcal{H}_{2g+1,q}}\sum_{n=0}^{g}q^{-n}\sum_{\substack{f \ \mathrm{monic}\\ \mathrm{deg}(f)=n \\ f\neq\square}}\chi_{D}(f)\ll(2q)^{g}.
\end{equation}
\item
\begin{equation}\label{eq:non2}
q^{-g}\sum_{D\in\mathcal{H}_{2g+1,q}}\sum_{m=0}^{g-1}\sum_{\substack{f \ \mathrm{monic}\\ \mathrm{deg}(f)=m \\ f\neq\square}}\chi_{D}(f)\ll (2q)^{g}.
\end{equation}
\end{enumerate}
\end{lemma}
\begin{proof}
We will establish the Part (1) of this Lemma. For Part (2) we have that the proof is analogous. We start with
\begin{multline}
\sum_{D\in\mathcal{H}_{2g+1,q}}\sum_{n=0}^{g}q^{-n}\sum_{\substack{f \ \mathrm{monic}\\ \mathrm{deg}(f)=m \\ f\neq\square}}\chi_{D}(f)\\
=\sum_{n=0}^{g}q^{-n}\sum_{\substack{f \ \mathrm{monic}\\ \mathrm{deg}(f)=n \\ f\neq\square}}\sum_{\substack{D \ \mathrm{monic} \\ \mathrm{deg}(D)=2g+1}}\sum_{\substack{A \ \mathrm{monic} \\ A^{2}\mid D}}\mu(A)\left(\frac{D}{f}\right)\ \ \ \ \ \ \ \ \ \ \ \ \ \ \ \ \ \ \ \ \ \ \ \ \ \ \ \ \ \ \ \ \ \ \ \ \ \ \ \ \ \ \ \ \ \ \ \ \ \ \ \ \\
=\sum_{n=0}^{g}q^{-n}\sum_{\substack{f \ \mathrm{monic}\\ \mathrm{deg}(f)=n \\ f\neq\square}}\sum_{\substack{A \ \mathrm{monic} \\ \mathrm{deg}(A)\leq g}}\mu(A)\left(\frac{A}{f}\right)^{2}\sum_{\substack{B \ \mathrm{monic} \\ \mathrm{deg}(B)=2g+1-2\mathrm{deg}(A)}}\left(\frac{B}{f}\right).\ \ \ \ \ \ \ \ \ \ \ \ \ \ \ \ \ \ \ \ \ \ \ \ \ \ \ \ \ \ \ \
\end{multline}
Now we note that $(B/f)$ is a nontrivial character since $f\neq\square$. So we can invoke the Lemma \ref{lem:RHfunctionfields} to get the following bound
\begin{equation}
\Bigg|\sum_{\substack{B \ \mathrm{monic} \\ \mathrm{deg}(B)=2g+1-2\mathrm{deg}(A)}}\left(\frac{B}{f}\right)\Bigg|\leq\binom{\mathrm{deg}(f)-1}{2g+1-2\mathrm{deg}(A)}q^{g+\tfrac{1}{2}-\mathrm{deg}(A)}
\end{equation}
if $2g+1-2\mathrm{deg}(A)<\mathrm{deg}(f)$, and the sum is zero otherwise. So we have,
\begin{multline}
\sum_{D\in\mathcal{H}_{2g+1,q}}\sum_{n=0}^{g}q^{-n}\sum_{\substack{f \ \mathrm{monic}\\ \mathrm{deg}(f)=n \\ f\neq\square}}\chi_{D}(f)\\
\leq\sum_{n=0}^{g}q^{-n}\sum_{\substack{f \ \mathrm{monic}\\ \mathrm{deg}(f)=n \\ f\neq\square}}\sum_{\substack{A \ \mathrm{monic} \\ \mathrm{deg}(A)\leq g}}\mu(A)\Bigg|\sum_{\substack{B \ \mathrm{monic} \\ \mathrm{deg}(B)=2g+1-2\mathrm{deg}(A)}}\left(\frac{B}{f}\right)\Bigg|\ \ \ \ \ \ \ \ \ \ \ \ \ \ \ \ \ \ \ \ \ \ \ \ \ \ \ \\
\leq\sum_{n=0}^{g}q^{-n}\sum_{\substack{f \ \mathrm{monic}\\ \mathrm{deg}(f)=n \\ f\neq\square}}\sum_{j=g+\tfrac{1}{2}-\tfrac{\mathrm{deg}(f)}{2}}^{g}\sum_{\substack{A \ \mathrm{monic} \\ \mathrm{deg}(A)=j}}\binom{\mathrm{deg}(f)-1}{2g+1-2\mathrm{deg}(A)}q^{g+\tfrac{1}{2}-\mathrm{deg}(A)}\ \ \ \ \ \ \ \ \ \ \ \ \ \ \ \ \ \ \\
\ll q^{g}\sum_{n=0}^{g}q^{-n}\sum_{\substack{f \ \mathrm{monic}\\ \mathrm{deg}(f)=n}}\sum_{j=g+\tfrac{1}{2}-\tfrac{\mathrm{deg}(f)}{2}}^{g}\binom{\mathrm{deg}(f)-1}{2g+1-2j}\ \ \ \ \ \ \ \ \ \ \ \ \ \ \ \ \ \ \ \ \ \ \ \ \ \ \ \ \ \ \ \ \ \ \ \ \ \ \ \ \ \ \ \ \ \ \ \ \ \ \ \ \ \ \\
\ll q^{g}\sum_{n=0}^{g}q^{-n}\sum_{\substack{f \ \mathrm{monic}\\ \mathrm{deg}(f)=n}}2^{\mathrm{deg}(f)-1}=q^{g}\sum_{n=0}^{g}2^{n}\ll q^{g}2^{g}.\ \ \ \ \ \ \ \ \ \ \ \ \ \ \ \ \ \ \ \ \ \ \ \ \ \ \ \ \ \ \ \ \ \ \ \ \ \ \ \ \ \ \ \ \ \ \ \ \ \ \ \ \ \ \nonumber
\end{multline}
\end{proof}

We now state and prove our next two lemmas.

\begin{lemma}\label{lem:lem2}
For $|D|=q^{2g+1}$ we have,
\begin{equation}
|D|\sum_{\substack{d \ \mathrm{monic} \\ \mathrm{deg}(d)>[g/2]}}\frac{\mu(d)}{|d|^{2}}\prod_{\substack{P \ \mathrm{monic} \\ \mathrm{irreducible} \\ P\mid d}}\frac{1}{|P|+1}\ll q^{g}.
\end{equation}
\end{lemma}
\begin{proof}
\begin{eqnarray}
|D|\sum_{\substack{d \ \mathrm{monic} \\ \mathrm{deg}(d)>[g/2]}}\frac{\mu(d)}{|d|^{2}}\prod_{P\mid d}\frac{1}{|P|+1} & \leq &  |D|\sum_{\substack{d \ \mathrm{monic} \\ \mathrm{deg}(d)>[g/2]}}\frac{1}{|d|^{2}}\prod_{P\mid d}\frac{1}{|P|}\nonumber\\
& = & |D|\sum_{h>[g/2]}q^{-2h}\ll |D|q^{-g}\nonumber\\
& \ll & q^{g}.\nonumber
\end{eqnarray}
\end{proof}

\begin{lemma}\label{lem:lem3}
We have,
\begin{multline}
\frac{|D|}{\zeta_{A}(2)}\sum_{\substack{d \ \mathrm{monic} \\ \mathrm{deg}(d)\leq[g/2]}}\frac{\mu(d)}{|d|}\prod_{P\mid d}\frac{1}{|P|+1}\left(\frac{q^{-\mathrm{deg}(d)}}{1-q^{-1}}\right)\\
=|D|\prod_{\substack{P \ \mathrm{monic} \\ \mathrm{irreducible}}}\left(1-\frac{1}{(|P|+1)|P|^{2}}\right)+O(q^{g}).
\end{multline}
\end{lemma}
\begin{proof}
\begin{multline}
\frac{|D|}{\zeta_{A}(2)}\sum_{\substack{d \ \mathrm{monic} \\ \mathrm{deg}(d)\leq[g/2]}}\frac{\mu(d)}{|d|}\prod_{P\mid d}\frac{1}{|P|+1}\left(\frac{q^{-\mathrm{deg}(d)}}{1-q^{-1}}\right)\\
=|D|\sum_{d \ \mathrm{monic}}\frac{\mu(d)}{|d|^{2}}\prod_{P\mid d}\frac{1}{|P|+1}-|D|\sum_{\substack{d \ \mathrm{monic}\\ \mathrm{deg}(d)>[g/2]}}\frac{\mu(d)}{|d|^{2}}\prod_{P\mid d}\frac{1}{|P|+1}.
\end{multline}
Writing the sum over all monic polynomials $d$ as an Euler product and using the Lemma \ref{lem:lem2} in the sum over $d$ such that $\mathrm{deg}(d)>[g/2]$ we obtain the desired lemma.\\
\end{proof}

Using the same ideas used in the proof of Lemmas \ref{lem:lem2} and \ref{lem:lem3} we can also prove the following lemmas.

\begin{lemma}
We have,
\begin{enumerate}
\item 
\begin{equation}
|D|\sum_{\substack{d \ \mathrm{monic} \\ \mathrm{deg}(d)>[g/2]}}\frac{\mu(d)}{|d|}\prod_{P\mid d}\frac{1}{|P|+1}\ll q^{g}.
\end{equation}
\item 
\begin{equation}
\frac{|D|q^{-g}q^{[(g-1)/2]+1}}{\zeta_{A}(2)(1-q)}\sum_{\substack{d \ \mathrm{monic} \\ \mathrm{deg}(d)>[(g-1)/2]}}\frac{\mu(d)}{|d|}\prod_{P\mid d}\frac{1}{|P|+1}\ll q^{g}.
\end{equation}
\end{enumerate}
\end{lemma}

\begin{lemma}\label{lem:lem5}
We have,
\begin{enumerate}
\item
\begin{multline}
\frac{|D|}{\zeta_{A}(2)}\sum_{\substack{d \ \mathrm{monic} \\ \mathrm{deg}(d)\leq[g/2]}}\frac{\mu(d)}{|d|}\prod_{P\mid d}\frac{1}{|P|+1}\left(\frac{q^{-[g/2]-1}}{1-q^{-1}}\right)\\
=|D|q^{-[g/2]-1}\prod_{\substack{P \ \mathrm{monic} \\ \mathrm{irreducible}}}\left(1-\frac{1}{|P|(|P|+1)}\right)+O(q^{g}).
\end{multline}
\item
\begin{multline}
\frac{|D|q^{-g}}{\zeta_{A}(2)}\sum_{\substack{d \ \mathrm{monic} \\ \mathrm{deg}(d)\leq[(g-1)/2]}}\frac{\mu(d)}{|d|}\prod_{P\mid d}\frac{1}{|P|+1}\left(\frac{q^{[(g-1)/2]+1}}{1-q}\right)\\
=\frac{|D|q^{-g}q^{[(g-1)/2]+1}}{\zeta_{A}(2)(1-q)}\prod_{\substack{P \ \mathrm{monic} \\ \mathrm{irreducible}}}\left(1-\frac{1}{|P|(|P|+1)}\right)+O(q^{g}).
\end{multline}
\end{enumerate}
\end{lemma}

We present now our last lemma,

\begin{lemma}\label{lem:lem6}
We have,
\begin{equation}
\frac{|D|q^{-g}}{\zeta_{A}(2)(1-q)}\sum_{\substack{d \ \mathrm{monic} \\ \mathrm{deg}(d)\leq[(g-1)/2]}}\mu(d)\prod_{P\mid d}\frac{1}{|P|+1}\ll gq^{g}.
\end{equation}
\end{lemma}
\begin{proof}
We have that,
\begin{eqnarray}
\frac{|D|q^{-g}}{\zeta_{A}(2)(1-q)}\sum_{\substack{d \ \mathrm{monic} \\ \mathrm{deg}(d)\leq[(g-1)/2]}}\mu(d)\prod_{P\mid d}\frac{1}{|P|+1}&\ll&|D|q^{-g}\sum_{\substack{d \ \mathrm{monic} \\ \mathrm{deg}(d)\leq[(g-1)/2]}}\frac{1}{|d|}\nonumber\\
&\ll&|D|q^{-g}([(g-1)/2]+1)\nonumber\\
&\ll& gq^{g}.\nonumber
\end{eqnarray}
\end{proof}

\section{The Main--Term}
The main result of this section is the following proposition, which will be used to establish the Main-Term of Theorem \ref{thm:mainthm1}.
\begin{proposition}\label{prop:2}
\begin{equation}
\sum_{\substack{D\in\mathcal{H}_{2g+1,q} \\ (D,l)=1}}1=\frac{|D|}{\zeta_{A}(2)\prod_{P|l}(1+|P|^{-1})}+O\left(\sqrt{|D|}\frac{\Phi(l)}{|l|}\right),
\end{equation}
\end{proposition}

We will need the following lemmas:
\begin{lemma}\label{lem:4}
Let $V_{d}=\{D\in\mathbb{F}_{q}[x]: D \ \mathrm{monic}, \ \mathrm{deg}(D)=d\}$. Then,
\begin{equation}
\#\{D\in V_{d}: (D,l)=1\}=q^{d}\frac{\Phi(l)}{|l|}.
\end{equation}
\end{lemma}
\begin{proof}
\begin{eqnarray}
\#\{D\in V_{d}:(D,l)=1\}& = &\sum_{\substack{D \ \mathrm{monic} \\ \mathrm{deg}(D)=d \\ (D,l)=1}}1=\sum_{\substack{D \ \mathrm{monic} \\ \mathrm{deg}(D)=d}}\sum_{h|(D,l)}\mu(h)\nonumber\\
& = & \sum_{h|l}\mu(h)\sum_{\substack{D \ \mathrm{monic} \\ \mathrm{deg}(D)=d \\ h|D}}1=\sum_{h|l}\mu(h)\sum_{\substack{m \ \mathrm{monic} \\ \mathrm{deg}(m)=d-\mathrm{deg}h}}1\nonumber\\
\label{eq:5.5}& = & q^{d}\prod_{\substack{P \\ P|l}}\left(1-\frac{1}{|P|}\right)=q^{d}\frac{\Phi(l)}{|l|}
\end{eqnarray}
where we used \cite[Proposition 2.4]{9} in \eqref{eq:5.5}.
\end{proof}

\begin{lemma}\label{lem:5}
We have,
\begin{equation}
\sum_{\substack{Q \ \mathrm{monic} \\ \mathrm{deg}(Q)>\frac{2g+1}{2} \\ (Q,l)=1}}\frac{\mu(Q)}{|Q|^{2}}\ll q^{-1/2}q^{-g}.
\end{equation}
\end{lemma}
\begin{proof}
\begin{eqnarray}
\sum_{\substack{Q \ \mathrm{monic} \\ \mathrm{deg}(Q)>\frac{2g+1}{2} \\ (Q,l)=1}}\frac{\mu(Q)}{|Q|^{2}} & \leq & \sum_{\substack{Q \ \mathrm{monic} \\ \mathrm{deg}(Q)>\frac{2g+1}{2} \\ (Q,l)=1}}\frac{1}{|Q|^{2}}\nonumber\\
& = & \sum_{n>\frac{2g+1}{2}}\frac{1}{q^{n}}\ll q^{-1/2}q^{-g}.
\end{eqnarray}
\end{proof}

\begin{lemma}\label{lem:6}
We have that,
\begin{equation}
\sum_{\substack{Q \ \mathrm{monic} \\ \mathrm{deg}(Q)\leq\frac{2g+1}{2} \\ (Q,l)=1}}\frac{\mu(Q)}{|Q|^{2}}=\frac{1}{\zeta_{A}(2)}\frac{1}{\prod_{P|l}(1-1/|P|^{2})}+O(q^{-1/2}q^{-g}).
\end{equation}
\end{lemma}
\begin{proof}

\begin{eqnarray}
\sum_{\substack{Q \ \mathrm{monic} \\ \mathrm{deg}(Q)\leq\frac{2g+1}{2} \\ (Q,l)=1}}\frac{\mu(Q)}{|Q|^{2}} & = & \sum_{\substack{Q \ \mathrm{monic} \\ (Q,l)=1}}\frac{\mu(Q)}{|Q|^{2}}-\sum_{\substack{Q \ \mathrm{monic} \\ \mathrm{deg}(Q)>\frac{2g+1}{2} \\ (Q,l)=1}}\frac{\mu(Q)}{|Q|^{2}}\nonumber\\
& = & \prod_{P\nmid l}\left(1-\frac{1}{|P|^{2}}\right)-\sum_{\substack{Q \ \mathrm{monic} \\ \mathrm{deg}(Q)>\frac{2g+1}{2} \\ (Q,l)=1}}\frac{\mu(Q)}{|Q|^{2}},
\end{eqnarray}
and
\begin{eqnarray}
\prod_{P\nmid l}\left(1-\frac{1}{|P|^{2}}\right) & = & \prod_{P}\left(1-\frac{1}{|P|^{2}}\right)\prod_{P\mid l}\left(1-\frac{1}{|P|^{2}}\right)^{-1}\nonumber\\
& = & \frac{1}{\zeta_{A}(2)}\frac{1}{\prod_{P\mid l}(1-1/|P|^{2})}.
\end{eqnarray}
Thus,
\begin{equation}
\sum_{\substack{Q \ \mathrm{monic} \\ \mathrm{deg}(Q)\leq\frac{2g+1}{2} \\ (Q,l)=1}}\frac{\mu(Q)}{|Q|^{2}}=\frac{1}{\zeta_{A}(2)}\frac{1}{\prod_{P|l}(1-1/|P|^{2})}-\sum_{\substack{Q \ \mathrm{monic} \\ \mathrm{deg}(Q)>\frac{2g+1}{2} \\ (Q,l)=1}}\frac{\mu(Q)}{|Q|^{2}},
\end{equation}
and using the estimate of Lemma \ref{lem:5} proves the result.
\end{proof}

\begin{proof}[Proof of Proposition \ref{prop:2}]
We have that
\begin{eqnarray}
\sum_{\substack{D\in\mathcal{H}_{2g+1,q} \\ (D,l)=1}}1 &=& \sum_{\substack{D\in V_{2g+1} \\ (D,l)=1}}\sum_{Q^{2}\mid D}\mu(Q)=\sum_{\substack{Q \ \mathrm{monic} \\ \mathrm{deg}(Q)\leq\frac{2g+1}{2} \\ (Q,l)=1}}\mu(Q)\sum_{\substack{D\in V_{2g+1-2\mathrm{deg}(Q)} \\ (D,l)=1}}1\nonumber\\
&=& \sum_{\substack{Q \ \mathrm{monic} \\ \mathrm{deg}(Q)\leq\frac{2g+1}{2} \\ (Q,l)=1}}\mu(Q)\#\{D\in V_{2g+1-2\mathrm{deg}(Q)}:(D,l)=1\}.
\end{eqnarray}
By Lemma \ref{lem:4}, we have,
\begin{eqnarray}
\sum_{\substack{D\in\mathcal{H}_{2g+1,q} \\ (D,l)=1}}1 & = & \sum_{\substack{Q \ \mathrm{monic} \\ \mathrm{deg}(Q)\leq\frac{2g+1}{2} \\ (Q,l)=1}}\mu(Q)q^{2g+1-2\mathrm{deg}(Q)}\frac{\Phi(l)}{|l|}\nonumber\\
& = & |D|\frac{\Phi(l)}{|l|}\sum_{\substack{Q \ \mathrm{monic} \\ \mathrm{deg}(Q)\leq\frac{2g+1}{2} \\ (Q,l)=1}}\frac{\mu(Q)}{|Q|^{2}}.
\end{eqnarray}
Invoking Lemma \ref{lem:6} we obtain,
\begin{eqnarray}
\sum_{\substack{D\in\mathcal{H}_{2g+1,q} \\ (D,l)=1}}1 & = & |D|\frac{\Phi(l)}{|l|}\left(\frac{1}{\zeta_{A}(2)}\frac{1}{\prod_{P|l}(1-1/|P|^{2})}+O(q^{-1/2}q^{-g})\right)\nonumber\\
& = & |D|\frac{\Phi(l)}{|l|}\frac{1}{\zeta_{A}(2)}\frac{1}{\prod_{P|l}(1-1/|P|^{2})}+O\left(|D|\frac{\Phi(l)}{|l|}q^{-1/2}q^{-g}\right),
\end{eqnarray}
and using $\displaystyle{\frac{\Phi(l)}{|l|}=\prod_{P|l}(1-|P|^{-1})}$, we end up with
\begin{equation}
\sum_{\substack{D\in\mathcal{H}_{2g+1,q} \\ (D,l)=1}}1=\frac{|D|}{\zeta_{A}(2)\prod_{P|l}(1+|P|^{-1})}+O\left(\sqrt{|D|}\frac{\Phi(l)}{|l|}\right),
\end{equation}
which proves Proposition \ref{prop:2}.
\end{proof}

\section{Proof of the Theorem 2}

Our argument in this section follows closely \cite{1}. From \eqref{eq:approx}, our main goal is to obtain an asymptotic formula for

\begin{multline}\label{eq:4.1}
\sum_{D\in\mathcal{H}_{2g+1,q}}L(1,\chi_{D})\\
=\sum_{D\in\mathcal{H}_{2g+1,q}}\sum_{n=0}^{g}\sum_{\substack{f \ \mathrm{monic} \\ \mathrm{deg}(f)=n}}\chi_{D}(f)q^{-n}+q^{-g}\sum_{D\in\mathcal{H}_{2g+1,q}}\sum_{m=0}^{g-1}\sum_{\substack{f \ \mathrm{monic} \\ \mathrm{deg}(f)=m}}\chi_{D}(f).
\end{multline}

We begin by establishing an asymptotic formula for the first term in the right--hand side of \eqref{eq:4.1}. 

\begin{multline}\label{eq:4.2}
\sum_{D\in\mathcal{H}_{2g+1,q}}\sum_{n=0}^{g}\sum_{\substack{f \ \mathrm{monic} \\ \mathrm{deg}(f)=n}}\chi_{D}(f)q^{-n}\\ 
=\sum_{n=0}^{g}q^{-n}\sum_{D\in\mathcal{H}_{2g+1,q}}\sum_{\substack{f \ \mathrm{monic} \\ \mathrm{deg}(f)=n \\ f=l^{2}}}\chi_{D}(f)+\sum_{n=0}^{g}q^{-n}\sum_{D\in\mathcal{H}_{2g+1,q}}\sum_{\substack{f \ \mathrm{monic} \\ \mathrm{deg}(f)=n \\ f\neq\square}}\chi_{D}(f)
\end{multline} 

Making use of the first part of Lemma \ref{lem:nonsquare} we can write \eqref{eq:4.2} as
\begin{equation}\label{eq:4.3}
\sum_{D\in\mathcal{H}_{2g+1,q}}\sum_{n=0}^{g}\sum_{\substack{f \ \mathrm{monic} \\ \mathrm{deg}(f)=n}}\chi_{D}(f)q^{-n}=\sum_{n=0}^{g}q^{-n}\sum_{D\in\mathcal{H}_{2g+1,q}}\sum_{\substack{f \ \mathrm{monic} \\ \mathrm{deg}(f)=n \\ f=l^{2}}}\chi_{D}(f)+O((2q)^{g}).
\end{equation}
For the square terms $f=l^{2}$ we make use of Proposition \ref{prop:2} and we end with

\begin{multline}
\label{eq:4.4}
\sum_{n=0}^{g}q^{-n}\sum_{D\in\mathcal{H}_{2g+1,q}}\sum_{\substack{f \ \mathrm{monic} \\ \mathrm{deg}(f)=n \\ f=l^{2}}}\chi_{D}(f)\\
=\frac{|D|}{\zeta_{A}(2)}\sum_{m=0}^{[g/2]}q^{-m}\sum_{\substack{d \ \mathrm{monic} \\ \mathrm{deg}(d)\leq m}}\frac{\mu(d)}{|d|}\prod_{P\mid d}\frac{1}{|P|+1}+O(q^{g/2})\ \ \ \ \ \ \ \ \ \ \ \ \ \ \ \ \ \ \ \ \ \ \ \ \ \ \ \ \ \ \ \ \ \ \ \ \ \ \ \ \ \ \ \ \ \ \ \ \\
=\frac{|D|}{\zeta_{A}(2)}\sum_{\substack{d \ \mathrm{monic} \\ \mathrm{deg}(d)\leq[g/2]}}\frac{\mu(d)}{|d|}\prod_{P\mid d}\frac{1}{|P|+1}\left(\frac{(q^{-1})^{\mathrm{deg}(d)}-(q^{-1})^{[g/2]+1}}{1-q^{-1}}\right)+O(q^{g/2}) \ \ \ \ \ \ \ \ \ \ \ \ \ \ \ \ \ \ \ \ \ \ \ \ \ \ \ \ \ \ \ \ \ \ \ \ \ \ \ \ \ \ \ \ \ \ \ \\
=\frac{|D|}{\zeta_{A}(2)}\sum_{\substack{d \ \mathrm{monic} \\ \mathrm{deg}(d)\leq[g/2]}}\frac{\mu(d)}{|d|}\prod_{P\mid d}\frac{1}{|P|+1}\left(\frac{q^{-\mathrm{deg}(d)}}{1-q^{-1}}\right)\ \ \ \ \ \ \ \ \ \ \ \ \ \ \ \ \ \ \ \ \ \ \ \ \ \ \ \ \ \ \ \ \ \ \ \ \ \ \ \ \ \ \ \ \ \ \ \\ 
-\frac{|D|}{\zeta_{A}(2)}\sum_{\substack{d \ \mathrm{monic} \\ \mathrm{deg}(d)\leq[g/2]}}\frac{\mu(d)}{|d|}\prod_{P\mid d}\frac{1}{|P|+1}\left(\frac{q^{-[g/2]+1}}{1-q^{-1}}\right)+O(q^{g/2}). 
\end{multline}

Now for the first term of \eqref{eq:4.4} we can use Lemma \ref{lem:lem3} and for the second term we can use Lemma \ref{lem:lem5} and so we end up with the following formula for the square terms,

\begin{multline}\label{eq:eq4.5}
\sum_{n=0}^{g}q^{-n}\sum_{D\in\mathcal{H}_{2g+1,q}}\sum_{\substack{f \ \mathrm{monic} \\ \mathrm{deg}(f)=n \\ f=l^{2}}}\chi_{D}(f)=|D|\prod_{\substack{P \ \mathrm{monic}\\ \mathrm{irreducible}}}\left(1-\frac{1}{|P|^{2}(|P|+1)}\right)\\
-|D|q^{-[g/2]-1}\prod_{\substack{P \ \mathrm{monic}\\ \mathrm{irreducible}}}\left(1-\frac{1}{|P|(|P|+1)}\right)+O(q^{g}).
\end{multline}

Substituting \eqref{eq:eq4.5} in \eqref{eq:4.3} we have that,

\begin{multline}\label{eq:4.6}
\sum_{D\in\mathcal{H}_{2g+1,q}}\sum_{n=0}^{g}\sum_{\substack{f \ \mathrm{monic} \\ \mathrm{deg}(f)=n}}\chi_{D}(f)q^{-n}=|D|\prod_{\substack{P \ \mathrm{monic}\\ \mathrm{irreducible}}}\left(1-\frac{1}{|P|^{2}(|P|+1)}\right)\\
-|D|q^{-[g/2]-1}\prod_{\substack{P \ \mathrm{monic}\\ \mathrm{irreducible}}}\left(1-\frac{1}{|P|(|P|+1)}\right)+O((2q)^{g}).
\end{multline}

For the second term in the right--hand side of \eqref{eq:4.1} we mimic the calculations above to end with,

\begin{multline}
\label{eq:4.7}
q^{-g}\sum_{D\in\mathcal{H}_{2g+1,q}}\sum_{m=0}^{g-1}\sum_{\substack{f \ \mathrm{monic} \\ \mathrm{deg}(f)=m}}\chi_{D}(f)\\
=\frac{|D|q^{-g}}{\zeta_{A}(2)}\sum_{\substack{d \ \mathrm{monic} \\ \mathrm{deg}(d)\leq[(g-1)/2]}}\frac{\mu(d)}{|d|}\prod_{P\mid d}\frac{1}{|P|+1}\sum_{\mathrm{deg}(d)\leq n\leq[(g-1)/2]}q^{n}\ \ \ \ \ \ \ \ \ \ \ \ \ \ \ \ \ \ \ \ \ \ \ \ \ \ \ \ \ \ \ \ \ \ \ \ \ \ \ \ \ \ \ \ \ \ \ \\
\ \ \ \ \ \ \ \ \ \ \ \ \ \ \ \ \ \ \ \ \ \ \ \ \ \ \ \ \ \ \ \ \ \ \ \ \ \ \ \ \ \ \ \ \ \ \ \ \ \ \ \ \ \ \ \ \ \ \ \ \ \ +O(q^{[(g-1)/2]})+O((2q)^{g})\\ \\
=\frac{|D|q^{-g}}{\zeta_{A}(2)}\sum_{\substack{d \ \mathrm{monic} \\ \mathrm{deg}(d)\leq[(g-1)/2]}}\frac{\mu(d)}{|d|}\prod_{P\mid d}\frac{1}{|P|+1}\left(\frac{q^{\mathrm{deg}(d)}}{1-q}\right)\ \ \ \ \ \ \ \ \ \ \ \ \ \ \ \ \ \ \ \ \ \ \ \ \ \ \ \ \ \ \ \ \ \ \ \ \ \ \ \ \ \ \ \ \ \ \ \\
-\frac{|D|q^{-g}}{\zeta_{A}(2)}\sum_{\substack{d \ \mathrm{monic} \\ \mathrm{deg}(d)\leq[(g-1)/2]}}\frac{\mu(d)}{|d|}\prod_{P\mid d}\frac{1}{|P|+1}\left(\frac{q^{[(g-1)/2]+1}}{1-q}\right)+O((2q)^{g}),
\end{multline}
where the error $O((2q)^{g})$ arises when we consider $f\neq\square$ and using part 2 of Lemma \ref{lem:nonsquare}.

For the first term in \eqref{eq:4.7} we use the bound given in Lemma \ref{lem:lem6} and for the second term we have,

\begin{multline}\label{eq:4.8}
\frac{|D|q^{-g}}{\zeta_{A}(2)}\sum_{\substack{d \ \mathrm{monic} \\ \mathrm{deg}(d)\leq[(g-1)/2]}}\frac{\mu(d)}{|d|}\prod_{P\mid d}\frac{1}{|P|+1}\left(\frac{q^{[(g-1)/2]+1}}{1-q}\right)\\
=\frac{|D|q^{-g}}{\zeta_{A}(2)}\left(\sum_{d \ \mathrm{monic}}\frac{\mu(d)}{|d|}\prod_{P\mid d}\frac{1}{|P|+1}\left(\frac{q^{[(g-1)/2]+1}}{1-q}\right)\right)\\
-\frac{|D|q^{-g}}{\zeta_{A}(2)}\left(\sum_{\substack{d \ \mathrm{monic} \\ \mathrm{deg}(d)>[(g-1)/2]}}\frac{\mu(d)}{|d|}\prod_{P\mid d}\frac{1}{|P|+1}\left(\frac{q^{[(g-1)/2]+1}}{1-q}\right)\right).
\end{multline}
And we can use part 2 of the Lemma \ref{lem:lem5}, and so we have that,

\begin{multline}
\frac{|D|q^{-g}}{\zeta_{A}(2)}\sum_{\substack{d \ \mathrm{monic} \\ \mathrm{deg}(d)\leq[(g-1)/2]}}\frac{\mu(d)}{|d|}\prod_{P\mid d}\frac{1}{|P|+1}\left(\frac{q^{[(g-1)/2]+1}}{1-q}\right)=\\
\frac{|D|q^{-g}q^{[(g-1)/2]+1}}{\zeta_{A}(2)(1-q)}\prod_{\substack{P \ \mathrm{monic} \\ \mathrm{irreducible}}}\left(1-\frac{1}{|P|(|P|+1)}\right)+O((2q)^{g}).
\end{multline}

So we can conclude that,
\begin{multline}\label{eq:4.10}
q^{-g}\sum_{D\in\mathcal{H}_{2g+1,q}}\sum_{m=0}^{g-1}\sum_{\substack{f \ \mathrm{monic} \\ \mathrm{deg}(f)=m}}\chi_{D}(f)=\\
-\frac{|D|q^{-g}q^{[(g-1)/2]+1}}{\zeta_{A}(2)(1-q)}\prod_{\substack{P \ \mathrm{monic} \\ \mathrm{irreducible}}}\left(1-\frac{1}{|P|(|P|+1)}\right)+O((2q)^{g}).
\end{multline}

Putting together the equations \eqref{eq:4.6} and \eqref{eq:4.10} and factoring $|D|$ we have that the proof of Theorem \ref{thm:mainthm1} is complete.
\begin{flushright}
$\square$
\end{flushright}

\section*{Acknowledgments}

I would like to thank Professor Jon Keating for introducing me to the subject and problems tackled in this paper, as well as for his useful advice during the course of the research. I would also like to thank Professors Brian Conrey, Ze\'{e}v Rudnick and Nina Snaith for numerous interesting discussions.

The author also wishes to thank the anonymous referees for valuable comments and suggestions that helped improve the presentation of the paper.






\end{document}